\documentclass{amsart}

\newtheorem{theorem}{Theorem}[section]
\newtheorem{lemma}[theorem]{Lemma}

\theoremstyle{definition}

\newtheorem{example}[theorem]{Example}

\theoremstyle{remark}
\newtheorem{remark}[theorem]{Remark}

\numberwithin{equation}{section}

\theoremstyle{Corollary}
\newtheorem{Corollary}[theorem]{Corollary}

\theoremstyle{Proposition}

\theoremstyle{Conjecture}
\newtheorem{Conjecture}[theorem]{Conjecture}

\usepackage{amsmath}
\usepackage{amsfonts,amssymb}

\usepackage{nth}
\usepackage{enumitem}

\begin{document}
\title{Orthogonal polynomials for the weakly equilibrium Cantor sets}
\author{G\H{o}kalp Alpan}
\address{Department of Mathematics, Bilkent University, 06800 Ankara, Turkey}

\email{gokalp@fen.bilkent.edu.tr}
\thanks{}

%    author two information
\author{Alexander Goncharov}
\address{Department of Mathematics, Bilkent University, 06800 Ankara, Turkey}

\email{goncha@fen.bilkent.edu.tr}
\thanks{}

%    \subjclass is required.

%\date{}

%\dedicatory{}

%    "Communicated by" -- provide editor's name; required.
%\commby{}

%    Abstract is required.
\begin{abstract}
	Let $K(\gamma)$ be the weakly equilibrium Cantor type set introduced in \cite{gonc}. It is proven that
	the monic orthogonal polynomials  $Q_{2^s}$ with respect to the equilibrium measure of $K(\gamma)$ coincide with the Chebyshev polynomials of the set. Procedures are suggested to find  $Q_{n}$ of all degrees
	and the corresponding Jacobi parameters. It is shown that the sequence of the Widom factors is bounded below.
\end{abstract}

\maketitle
\section{Introduction}

Cantor sets appear as supports of spectral measures for the Jacobi operators in some natural situations.
For example, the Schr\"odinger operator with an almost periodic potential (for example, the almost Mathieu operator)
typically has such spectrum (see e.g. the review \cite{simon2}). On the other hand, we can meet Julia sets $K$ of the Cantor type in the theory of orthogonal polynomials (OP) with respect to the equilibrium measure $\mu_K$ (see e.g \cite{Barnsley3,Barnsley4,Belissard}).
Whereas, due to B. Simon et al., there is a comprehensive theory for the finite gap Jacobi matrices, until now
there is no such theory for purely singular continuous measures. In order to analyze the infinite gap Jacobi matrices, the class of Parreau-Widom sets was suggested (see e.g. \cite{christiansen}). This notion includes Cantor sets of positive Lebesgue measure, so the Szeg\"o condition can be applied in this case. Recently numerous attempts were undertaken to investigate the spectral problem for self-similar measures with Cantor supports. By means of computational methods,
some conjectures (see e.g. \cite{Heilman,kruger}) about the asymptotic behavior of the Jacobi parameters and other related values were posed.

Here we consider one more example of a family of OP with respect to $\mu_{K(\gamma)}$
on the Cantor set $K(\gamma),$ introduced in \cite{gonc}. The set  $K(\gamma)$ depends on a parameter
$\gamma =(\gamma_s)_{s=1}^\infty$ and is constructed by means of a recurrent procedure. If $\gamma_s$ are not very small then the set $K(\gamma)$ is not polar.
At least in known cases, the set  $K(\gamma)$ is dimensional and, by \cite{alpan}, $\mu_{K(\gamma)}$ is
mutually absolutely continuous with the corresponding Cantor-Hausdorff measure. This is not valid
for geometrically symmetric zero Lebesgue measure Cantor sets, where, by \cite{Makarov} and followers, these measures are mutually singular.

In Section 2 we show that the monic OP $Q_{2^s}$ coincide with the corresponding Chebyshev polynomials.
In Sections 3 and 4 we suggest a procedure to find $Q_{n}$ for $n\ne 2^s.$ This allows to analyze
the asymptotics of the Jacobi parameters $(a_n)_{n=1}^{\infty}$.
Since $Cap(K(\gamma))$ is known, we estimate (Section 5) the {\it Widom factors}
$ W_n := \frac{a_1\cdots a_n}{Cap(K(\gamma))^n}$ and check the Widom condition, that
characterizes the Szeg\H{o} class of Jacobi matrices in the finite gap case.
In the last section we discuss a possible version of the Szeg\H{o} condition for
singular continuous measures. At least for $\gamma_s\leq 1/6, s\in \mathbb{N},$
the Lebesgue measure of the set  $K(\gamma)$ is zero, so it is not a Parreau-Widom set.

For the basic concepts of the theory of logarithmic potential see e.g \cite{ransford}, $\log$ denotes the natural
logarithm, $Cap(\cdot)$ stands for the logarithmic capacity, $0^0:=1.$

\section{Orthogonal Polynomials}

Given a sequence $\gamma =(\gamma_s)_{s=1}^\infty$ with $0<\gamma_s<1/4$ define $r_0=1$ and $r_s=\gamma_s r_{s-1}^2$. Let
\begin{equation}\label{poly1}
P_1:=x-1 \,\,\,\mbox{and} \,\,\, P_{2^{s+1}}(x):= P_{2^s}(x)\cdot(P_{2^s}(x)+r_s)
\end{equation}
for $s\in  \mathbb{N}_0$ in a recursive fashion.
Thus, $P_2(x)=x\cdot(x-1)$ for each  $\gamma,$ whereas, for $s\geq 2,$ the polynomial $P_{2^s}$ essentially
depends on the parameter $\gamma.$
For $s\in  \mathbb{N}_0$ consider
$$
\label{roro}E_s = \{x\in\mathbb{R}:\,  P_{2^{s+1}}(x)\leq 0\}= \left(\frac{2}{r_s}P_{2^s}+1\right)^{-1}\left([-1,1]\right) = \cup_{j=1}^{2^s} I_{j,s},
$$
where $\mathbb{N}_0:=\mathbb{N}\cup \{0\}$ and $I_{j,s}$ are closed {\it basic} intervals of the $s-$th level which are necessarily disjoint. Then $E_{s+1}\subset E_{s}$ and the set $K(\gamma):=\cap_{s=0}^{\infty}E_s$ is a Cantor set by Lemma 2 in \cite{gonc}.

Let $l_{i,s}$ stand for the length of $I_{j,s}$ when we enumerate them from the left to the right.
By Lemma 6 in \cite{gonc},
$$\gamma_1\ldots\gamma_s<l_{i,s}<\exp{\left(16\sum_{k=1}^s \gamma_k\right)}\gamma_1\ldots\gamma_s,\,\,\,\,\,\, 1\leq i\leq 2^s,$$
provided $\gamma_k\leq 1/32$ for all $k.$ Therefore, the Lebesgue measure $|K(\gamma)|$ of the set with this condition is zero and, by \cite{christiansen}, $K(\gamma)$ is not a Parreau-Widom set. For the definition of Parreau-Widom sets see e.g. \cite{christiansen}. In Section 4 we will show that $|K(\gamma)|=0$ as well if $\gamma_k\leq 1/6$ for all $k.$

On the other hand, by choosing $(\gamma_k)_{k=1}^\infty$ sufficiently close to $1/4$, we can obtain Cantor sets with positive Lebesgue measure.  What is more, in the limit case, when all $\gamma_k =1/4,$ we get $E_s=[0,1]$ for all $s$ and  $K(\gamma)=[0,1]$ (see Example 1 in  \cite{gonc}).

In addition, by Corollary 1 in  \cite{gonc}, $Cap(K(\gamma))=\exp{\left(\sum_{k=1}^{\infty} 2^{-k}\log{\gamma_k} \right)}.$ In the paper we assume $Cap(K(\gamma))>0.$
Let  $\mu_{K(\gamma)}$ denote the equilibrium measure on the set, $||\cdot||$ be the norm in the
corresponding Hilbert space.
From Corollary 3.2 in \cite{alpan} we have $\mu_{K(\gamma)}(I_{j,s})=2^{-s}$ for all $s$ and $ 1\leq j\leq 2^s,$
provided $\gamma_k\leq 1/32$ for all $k.$

From now on, by $Q_n$ we denote the monic orthogonal polynomial of degree $n\in \mathbb{N}$ with respect to
$\mu_{K(\gamma)}.$ The main result of this section is that, for $n=2^s$ with $s\in  \mathbb{N}_0,$ the polynomial
$Q_n$ coincides with the corresponding Chebyshev polynomial for $K(\gamma).$
The next two theorems will play a crucial role.

\begin{theorem}[{\cite{gonc}}, Prop.1] \label{cheb} For each $s\in  \mathbb{N}_0$ the polynomial $P_{2^s}+r_s/2$ is the Chebyshev polynomial for $K(\gamma).$
\end{theorem}
\begin{remark}  Only the values $s\in  \mathbb{N}$ were considered in \cite{gonc}. But, clearly, for $s=0$
the polynomial $P_1+1/2=x-1/2$ is Chebyshev.
\end{remark}
\begin{remark}
Since real polynomials are considered here and the alternating set for $P_{2^s}+r_s/2$ consists of $2^s+1$
points, the Chebyshev property of this polynomial follows by the alternance argument.
\end{remark}

\begin{theorem}[{\cite{saff}}, T.3.6]\label{totcheb} Let $K\subset\mathbb{R}$ be a non-polar compact set. Then the normalized counting measures on the zeros of the Chebyshev polynomials converge to the equilibrium measure of $K$ in the weak-star topology.
\end{theorem}

For $s\in  \mathbb{N},$ the polynomial $P_{2^s}+{r_s}\mathbin{/}{2}$ has simple real zeros $(x_k)_{k=1}^{2^s}$
which are symmetric about $x=1/2$. Let us denote by $\nu_s$  the normalized counting measure at these points,
that is $\nu_s=2^{-s}\sum_{k=1}^{2^s} \delta_{x_k}.$

\begin{lemma} \label{lem1} Let $s>m$ with $s,m\in \mathbb{N}_0$. Then
$ \int\left(P_{2^m}+\frac{r_m}{2}\right)d\nu_s=0.$

\end{lemma}

\begin{proof} For $m=0$ we have the result by symmetry. Suppose $m\geq 1.$ By \eqref{poly1}, at the points
$(x_k)_{k=1}^{2^s}$ we have
$$P_{2^s}+\frac{r_s}{2}=(P_{2^{s-1}})^2+r_{s-1}P_{2^{s-1}}+\frac{r_s}{2}=0.$$
The discriminant of the equation is positive. Therefore, the roots satisfy $$(P_{2^{s-1}}+\alpha_{s-1}^1)(P_{2^{s-1}}+\alpha_{s-1}^2)=0,$$ where $\alpha_{s-1}^1+\alpha_{s-1}^2=r_{s-1}$ and $0<\alpha_{s-1}^1,\alpha_{s-1}^2<r_{s-1}$. Thus, a half of the points satisfy $P_{2^{s-1}}+\alpha_{s-1}^1=0$ while the other half satisfy $P_{2^{s-1}}+\alpha_{s-1}^2=0.$

Rewriting the equation $P_{2^{s-1}}+\alpha_{s-1}^1=0$, we see that $$P_{2^{s-2}}^2+r_{s-2}P_{2^{s-2}}+\alpha_{s-1}^1=0.$$  Since $r_{s-2}^2>4 r_{s-1}>4\alpha_{s-1}^1$, this yields $$(P_{2^{s-2}}+\alpha_{s-2}^1)(P_{2^{s-2}}+\alpha_{s-2}^2)=0$$ with $\alpha_{s-2}^1+\alpha_{s-2}^2=r_{s-2}$ and $0<\alpha_{s-2}^1,\alpha_{s-2}^2<r_{s-2}.$ By the same argument, the second half of the roots
satisfy $$(P_{2^{s-2}}+\alpha_{s-2}^3)(P_{2^{s-2}}+\alpha_{s-2}^4)=0$$  with $\alpha_{s-2}^3+\alpha_{s-2}^4=r_{s-2}$ and $0<\alpha_{s-2}^3,\alpha_{s-2}^4<r_{s-2}.$

Since at each step $r_{i-1}^2>4 r_{i}$ we can continue this procedure until obtaining $P_{2^{m+1}}.$ So we can
decompose the Chebyshev nodes $(x_k)_{k=1}^{2^s}$ into $2^{s-m-1}$ groups.
All $2^{m+1}$ nodes from the $i-$th group satisfy $$P_{2^{m+1}}+\alpha_{m+1}^{i}=0,\,\,\, 0<\alpha_{m+1}^{i}<r_{m+1}.$$
By using these $2^{s-m-1}$ equations we finally obtain $$(P_{2^{m}}+\alpha_m^{2i-1})(P_{2^{m}}+\alpha_m^{2i})=0$$ where $\alpha_m^{2i-1}+\alpha_m^{2i}=r_m.$ Thus, given fixed $i-$th group, for $2^m$ points from the group we have
$P_{2^{m}}=-\alpha_m^{2i-1},$ whereas for another half, $P_{2^{m}}=-\alpha_m^{2i}.$ 
Consequently, we have $$\displaystyle\int\left(P_{2^m}+\frac{r_m}{2}\right)d\nu_s=\displaystyle\int P_{2^m} d\nu_s+\frac{r_m}{2}=\frac{\sum_{i=1}^{2^{s-m-1}} 2^m(-\alpha_m^{2i-1}-\alpha_m^{2i})}{2^s}+\frac{r_m}{2}=0.$$ 
\end{proof}

\begin{lemma}\label{lem2}Let $0\leq i_1<i_2<\ldots i_n<s.$  Then
\begin{enumerate}[label=(\alph*)]
\item \begin{eqnarray*}
\int P_{2^{i_1}}P_{2^{i_2}}\ldots P_{2^{i_n}}d\nu_s&=&
\int P_{2^{i_1}}d\nu_s \int P_{2^{i_2}}d\nu_s\ldots \int P_{2^{i_n}}d\nu_s\\
&=&
(-1)^n\,\prod_{k=1}^n \frac{r_{i_k}}{2}.
\end{eqnarray*}
\item\label{part2}$\displaystyle\int\left(P_{2^{i_1}}+\frac{r_{i_1}}{2}\right)\left(P_{2^{i_2}}+\frac{r_{i_2}}{2}\right)\ldots \left(P_{2^{i_n}}+\frac{r_{i_n}}{2}\right)d\nu_s=0.$
\end{enumerate}
\end{lemma}
\begin{proof}
\begin{enumerate}[label=(\alph*)]
\item Suppose that $i_1\geq 1.$ As above, we can decompose the nodes $(x_k)_{k=1}^{2^s}$
into $2^{s-i_1-1}$ equal groups such that the nodes from the $j-$th group satisfy an equation
$$(P_{2^{i_1}}+\alpha_{i_1}^{2j-1})(P_{2^{i_1}}+\alpha_{i_1}^{2j})=0$$ with $\alpha_{i_1}^{2j-1}+\alpha_{i_1}^{2j}=r_{i_1}.$ If, on some set,
$(P_{2^k}+\alpha)(P_{2^k}+\beta)=0$ with $\alpha+\beta=r_k,$ then
$P_{2^{k+1}}=P_{2^k}^2+  P_{2^k}\,r_k=-\alpha \beta$ and each of the next polynomials $P_{2^{k+i}}$
is constant on this set. Therefore the function $P_{2^{i_2}}\ldots P_{2^{i_n}}$ takes the same
value for all $x_k$ from the  $j-$th group. This allows to apply the argument of Lemma \ref{lem1}:

$$\displaystyle\int P_{2^{i_1}}P_{2^{i_2}}\ldots P_{2^{i_n}}d\nu_s=-\frac{r_{i_1}}{2}\displaystyle\int P_{2^{i_2}}P_{2^{i_3}}\ldots P_{2^{i_{n}}}d\nu_s.$$
This equality is valid also for $i_1=0$ since $$\displaystyle \int\left(P_{1}+\frac{1}{2}\right)P_{2^{i_2}}\ldots P_{2^{i_n}}d\nu_s=0,$$
by symmetry. Proceeding this way, the result follows, since $ -r_m/2= \int P_{2^m}d\nu_s,$ by Lemma \ref{lem1}.\\

\item Opening the parentheses yields
$$\int P_{2^{i_1}}P_{2^{i_2}}\ldots P_{2^{i_n}}d\nu_s + \sum_{k=1}^n \frac{r_{i_k}}{2} \int \prod_{j\ne k}
 P_{2^{i_j}}d\nu_s + \cdots + \prod_{k=1}^n \frac{r_{i_k}}{2}.$$ By Lemma \ref{lem1} and part $(a)$, this is
 $$  \prod_{k=1}^n \frac{r_{i_k}}{2} \cdot \sum_{k=0}^n \binom{n}{k}\,(-1)^{n-k}=0.$$ 
 \end{enumerate}
\end{proof}

\begin{remark} We can use $\mu_{K(\gamma)}$ instead of $\nu_s$ in Lemma \ref{lem1} and Lemma \ref{lem2}, since,
 by Theorem \ref{totcheb}, $\nu_s {\to} \mu_{K(\gamma)}$ in the weak-star topology.
\end{remark}

\begin{theorem}\label{main1} The monic orthogonal polynomial $Q_{2^s}$ with respect to
the equilibrium measure $\mu_{K(\gamma)}$ coincides with the corresponding
Chebyshev polynomial $P_{2^s}+ r_s/2$ for all  $s\in \mathbb{N}_0$.
\end{theorem}
\begin{proof}
 For $s=0$ we have the result by symmetry. Let $s\geq 1.$ Each polynomial $P$ of degree less than $2^s$
 is a linear combination of polynomials of the type
$$\left(P_{2^{s-1}}+\frac{r_{s-1}}{2}\right)^{n_{s-1}}\ldots \left(P_{2}+\frac{r_1}{2}\right)^{n_1}
\left( x-\frac{1}{2} \right)^{n_0}$$
with  $n_i\in\{0,1\}.$ By Lemma \ref{lem2}, $P_{2^{s}}+{r_{s}}\mathbin{/}{2}$ is orthogonal to all
polynomials of degree less than $2^s,$ so it is $Q_{2^s}$.
\end{proof}

By \eqref{poly1}, we immediately have
\begin{Corollary}\label{cor1}
 $Q_{2^{s+1}}= Q_{2^s}^2- (1-2\,\gamma_{s+1})\,r_s^2/4$ for  $s\in \mathbb{N}_0$.
 \end{Corollary}

\section{Some products of orthogonal polynomials}
So far we only obtain $2^s$ degree orthogonal polynomials. We try to find $Q_n$
for other degrees.  By Corollary \ref{cor1}, since  $\int Q_{2^{s+1}}d\mu_{K(\gamma)}=0,$ we have
 \begin{equation} \label{norm1}
||Q_{2^s}||^2= \int Q_{2^s}^2d\mu_{K(\gamma)}=(1-2\,\gamma_{s+1})\,r_s^2/4
 \end{equation}

  and
\begin{equation} \label{norm2}
Q_{2^{s+1}}=Q_{2^s}^2-||Q_{2^s}||^2,\,\,\forall s\in \mathbb{N}_0.
\end{equation}

Our next goal is to evaluate $\int A\,d\mu_{K(\gamma)}$ for {\it A-polynomial} of the form
\begin{equation} \label{A}
A=(Q_{2^{s_n}})^{i_n}(Q_{2^{s_{n-1}}})^{i_{n-1}}\ldots (Q_{2^{s_{1}}})^{i_{1}},
\end{equation}
where $s_n>s_{n-1}>\ldots>s_1>0$ and $i_1,i_2,\ldots,i_n\in\{1,2\}$.\\

The next lemma is basically a consequence of \eqref{norm2}.

\begin{lemma}\label{int1}Let $A$ be a polynomial satisfying \eqref{A}. Then the following propositions hold:
\begin{enumerate}[label=(\alph*)]
\item If $i_n=2$ then $\displaystyle\int A\, d\mu_{K(\gamma)}=
\|Q_{2^{s_n}}\|^{2}\int Q_{2^{s_{n-1}}}^{\,i_{n-1}}\cdots Q_{2^{s_{1}}}^{\,i_{1}}d\mu_{K(\gamma)}.$\\
\item If $i_n=i_{n-1}=\ldots=i_{n-k}=1$ and $i_{n-k-1}=2$ where $k\in \{0,1,\ldots,n-2\},$ then
$$\displaystyle \int A\, d\mu_{K(\gamma)}= \|Q_{2^{s_n}}\|^{2} \int Q_{2^{s_{n-k-2}}}^{\,i_{n-k-2}} \cdots Q_{2^{s_{1}}}^{\,i_{1}} \, d\mu_{K(\gamma)},$$ provided $s_n=s_{n-1}+1=s_{n-2}+2=\cdots=s_{n-k-1}+k+1.$\\
\item If $i_k=1$ and $s_k\geq s_{k-1}+2$ for some $k\in \{2,3,\ldots,n\},$ then  $\displaystyle\int A\, d\mu_{K(\gamma)}=0.$\\
\item  If $i_1=i_2=\cdots =i_k=1$ for some $k\in \{1,2,\cdots,n\},$ then $\displaystyle\int A\, d\mu_{K(\gamma)}=0.$
\end{enumerate}
\end{lemma}
\begin{proof}
\begin{enumerate}[label=(\alph*)]
\item Using \eqref{norm2}, we have $ Q_{2^{s_n}}^{2}= Q_{2^{s_{n}+1}}+\|Q_{2^{s_n}}\|^{2}.$ The result easily follows since the degree of $Q_{2^{s_{n-1}}}^{\,i_{n-1}}\cdots Q_{2^{s_{1}}}^{\,i_{1}}$ is less than
 $2^{s_n+1}$.\\
\item Here $A= Q_{2^{s_n}}\,Q_{2^{s_n-1}} \cdots Q_{2^{s_n-k}}Q_{2^{s_n-k-1}}^2 \cdot P$ with $P=Q_{2^{s_{n-k-2}}}^{\,i_{n-k-2}} \cdots Q_{2^{s_{1}}}^{\,i_{1}}.$ Observe that the degrees of the first $k+2$
polynomials are consecutive and $\deg \,P <2^{s_{n-k-2}+2}\leq 2^{s_{n-k}}.$
We apply  \eqref{norm2} repeatedly. First, $\,Q_{2^{s_n-k-1}}^2 = Q_{2^{s_n-k}} +||Q_{2^{s_n-k-1}}||^2.$ Then $Q_{2^{s_n-k}}\, Q_{2^{s_n-k-1}}^2=Q_{2^{s_n-k+1}}+||Q_{2^{s_n-k}}||^2+ ||Q_{2^{s_n-k-1}}||^2\cdot Q_{2^{s_n-k}}.$ After $k+1$ steps we write $A$ in the form $ ( Q_{2^{s_{n}+1}}+\|Q_{2^{s_n}}\|^{2}+{\mathcal L})\,P,$ where ${\mathcal L}$ is a linear combination of the polynomials
$Q_{2^{s_n}},$ $\,Q_{2^{s_n}} Q_{2^{s_n-1}}, \cdots , Q_{2^{s_n}} Q_{2^{s_n-1}}\cdots Q_{2^{s_n-k}}.$
By orthogonality, all terms vanish after integration, except
$\|Q_{2^{s_n}}\|^{2}\,P,$ which is the desired conclusion.\\
\item Let us take the maximal $k$ with such property. Repeated application of $(a)$ and $(b)$ enables us to reduce $\int A\, d\mu_{K(\gamma)}$ to $C\,\int A_1\, d\mu_{K(\gamma)}$ with $C >0$ and
$A_1=Q_{2^{s_m}}\cdots Q_{2^{s_k}} \cdot R,$ where $R= Q_{2^{s_{k-1}}}^{i_{k-1}} \cdots Q_{2^{s_{1}}}^{i_{1}}$
with  $\deg \,R <2^{s_{k-1}+2}\leq 2^{s_{k}}.$
Comparing the degrees gives the result.\\

\item Similarly, $\int A\,d\mu_{K(\gamma)}= C\cdot \int Q_{2^{s_k}} \cdots Q_{2^{s_1}}  d\mu_{K(\gamma)}=0.$
\end{enumerate}
\end{proof}

\begin{theorem}\label{int2}
For $A-$polynomial given in \eqref{A} with $s_0:=-1, i_{n+1}:=2$, let $c_k=(i_k-1)^{s_k-s_{k-1}-1}$ and
$c=\prod_{k=1}^n c_k.$ Then $\int A\,d\mu_{K(\gamma)}= c\cdot \prod_{k=1}^n ||Q_{2^k}||^{2(i_{k+1}-1)}.$
\end{theorem}
\begin{proof}
First we remark that $c\in\{0,1\}$ and $c=0$ if and only if $i_1=1$ or $i_k=1$ for some $k\in \{2,3,\cdots, n\}$
with $s_k>s_{k-1}+1,$ that is just in the cases $(c)$ and $(d)$ above.

Secondly, the procedures $(a)-(d)$ of Lemma \ref{int1} allow to find  $\int A\,d\mu_{K(\gamma)}$ for all admissible values of $(i_k)_{k=1}^n$ and $(s_k)_{k=1}^n.$ Consider the string ${\mathcal I}= \{i_n, i_{n-1}, \cdots, i_1\}.$ If $i_n=2$ then we use the procedure $(a)$. Suppose $i_n=1.$ If all other $i_k=1$ then  $\int A\,d\mu_{K(\gamma)}=0,$
by $(d)$. Otherwise we have a string $\{1, \cdots, 1, 2, \cdots\}.$ Here we check
the values $s_k$. If they are consecutive (including the first appearance of 2), then we use $(b)$
and get  the term $\|Q_{2^{s_n}}\|^{2(i_{n+1}-1)}$ in front of the new integral. Otherwise, by $(c)$,
the integral is zero.

  Thus, in the case of non zero integral, we decompose  ${\mathcal I}$
 into substrings of the types $\{2\}, \,\{1, 2\}, \cdots, \{1, \cdots, 1, 2\}$ without gaps between consecutive $s_k$ in the same substring. If there are two or more  consecutive substrings  $\{2\}$ then gaps between
 corresponding  $s_k$ are allowed. Let $k$ with $k<n$ be the index of the first term in
 $ {\mathcal I}_k=\{1, \cdots, 1, 2\}.$ Then $i_{k+1}=2$ and, by $(b)$, $ {\mathcal I}_k$ brings in the term $||Q_{2^k}||^2$ into the common product.
 \end{proof}

 \begin{Corollary}\label{cor2}
 For $A-$polynomial given in \eqref{A}, let $A=A_1\cdot Q_{2^{s_{1}}}^{i_{1}},$ so $A_1$ contains all terms
 of $A$ except the last. Suppose $i_1=i_2=2.$ Then
 $\int A\,d\mu_{K(\gamma)}= ||Q_{2^{s_{1}}}||^2 \int A_1\,d\mu_{K(\gamma)}.$
 \end{Corollary}

We will represent $Q_n$ in terms of {\it B-polynomials} that are defined, for $2^{m}\leq n <2^{m+1}$ with $m\in\mathbb{N}_0,$ as
$$B_n= (Q_{2^{m}})^{i_{m}}(Q_{2^{m-1}})^{i_{m-1}}\ldots(Q_1)^{i_{1}},$$
where $i_k\in \{0,1\}$ is the $k-$th coefficient in the binary representation $n= i_m\,2^m+ \cdots +i_0.$

Thus, $B_n$ is a monic polynomial of degree $n.$  The polynomials $B_{(2k+1)\cdot2^s}$ and
$B_{(2j+1)\cdot2^m}$ are orthogonal for all $j,k,m,s \in \mathbb{N}_0$ with $s\ne m.$ Indeed,
if $\min\{m,s\}=0$ then $\int B_{(2k+1)\cdot2^s}\,B_{(2j+1)\cdot2^m}\,d\mu_{K(\gamma)} =0,$
since one polynomial is symmetric about $x=1/2,$ whereas another is antisymmetric.
Otherwise we use Lemma \ref{int1} $(d)$. By $(a)$, we have
$$||B_n||^2= \prod_{k=0}^m ||Q_{2^{k}}||^{2i_k}=\prod_{k=0, i_k\ne 0}^m ||Q_{2^{k}}||^2.$$

\begin{theorem}\label{Bpol} For each $n\in \mathbb{N},$ let  $n=2^s(2k+1),$ the polynomial $Q_n$ has a unique representation as a linear combination of $B_{2^s}, B_{3\cdot2^s}\ldots,B_{(2k-1)\cdot2^s}, B_{(2k+1)\cdot2^s}.$
\end{theorem}
\begin{proof} Consider $P=a_0 B_{2^s}+a_1 B_{3\cdot2^s}+\ldots+a_{k-1}B_{(2k-1)\cdot2^s}+B_{(2k+1)\cdot2^s},$
where $(a_j)_{j=0}^{k-1}$ are chosen such that $P$ is orthogonal to all $B_{(2j+1)2^s}$ with $j=0,1,\ldots, k-1.$
This gives a system of $k$ linear equations with $k$ unknowns $(a_j)_{j=0}^{k-1}.$ The determinant of this
system is the Gram determinant of linearly independent functions $(B_{(2j+1)2^s})_{j=0}^{k-1}.$
Therefore it is positive and the system has a unique solution. In addition, as was remarked above,
$P$ is orthogonal to all $B_{(2j+1)\cdot2^m}$ with $m\ne s.$ Thus, $P$ is a monic polynomial of
degree $n$ that is orthogonal to all polynomials of degree $<n,$ so $P= Q_n.$
\end{proof}

\begin{Corollary}\label{cor3}
 The polynomial $Q_{2^s(2k+1)}$ is a linear combination of products of the type
 $Q_{2^{s_m}}\,Q_{2^{s_{m-1}}} \cdots Q_{2^s},$ so the smallest degree of $Q_{2^{s_j}}$
 in every product is $2^s$.
\end{Corollary}

To illustrate the theorem, we consider, for given $s\in\mathbb{N}_0,$ the easiest cases with $k\leq 2$.
Clearly, $Q_{2^s}=B_{2^s}.$ Since $B_{3\cdot2^s}=Q_{2^s}Q_{2^{s+1}},$ we take
$Q_{3\cdot2^s}=a_0 Q_{2^s}+Q_{2^{s+1}} Q_{2^s},$ where $a_0$ is such that 
$\int Q_{3\cdot2^s}\,Q_{2^s}\,d\mu_{K(\gamma)}=0.$
 By Lemma \ref{int1},
$$Q_{3\cdot2^s}=Q_{2^{s+1}}Q_{2^s}-\frac{\|Q_{2^{s+1}}\|^{2}}{\|Q_{2^{s}}\|^{2}}Q_{2^s}.$$

Similarly, $B_{5\cdot2^s}=Q_{2^s}Q_{2^{s+2}}$ and $Q_{5\cdot2^s}=a_0 Q_{2^s}+a_1 Q_{2^{s+1}}Q_{2^s}
+ Q_{2^s}Q_{2^{s+2}}$ with
$$ a_0=\frac{||Q_{2^{s+2}}||^{2}}{||Q_{2^{s}}||^4 - ||Q_{2^{s+1}}||^2},\,\,\,
a_1= -a_0 \,\, \frac{||Q_{2^s}||^{2}}{\|Q_{2^{s+1}}\|^{2}}.$$

Using \eqref{norm1}, all coefficient can be expressed only in terms of $(\gamma_k)_{k=1}^\infty.$
As $k$ gets bigger, the complexity of calculations increases.

\begin{remark} In general, the polynomial $Q_n$ is not Chebyshev. For example,
$Q_3=Q_1(Q_2+a_0)$ with $a_0=- \frac{(1-2\gamma_2) \gamma_1^2}{1-2\gamma_1}.$ At least for small $\gamma_1,$ the polynomial $Q_3(x)=(x-1/2)(x^2-x+\gamma_1/2+a_0)$ increases on the first basic interval $I_{1,1}=[0, l_{1,1}].$
Here, $l_{1,1}$ is the first solution of $P_2=-r_1,$ so $l_{1,1}=(1-\sqrt{1-4\gamma_1})/2.$ If $Q_3$ is
the Chebyshev polynomial then, by the alternance argument,  $Q_3(l_{1,1})=Q_3(1),$ but it is not the case.
\end{remark}

\section{Jacobi parameters}

Since $\mu_{K(\gamma)}$ is supported on the real line, the polynomials $(Q_n)_{n=0}^{\infty}$
satisfy a three-term recurrence relation
$$Q_{n+1}(x) = (x- b_{n+1})Q_{n}(x) - a_n^2 \, Q_{n-1}(x),\,\,\,\,\,\,\,\,n \in \mathbb{N}_0. $$
The recurrence starts from $Q_{-1}:=0$ and $Q_0=1.$ The Jacobi parameters
$\{a_n,b_n\}_{n=1}^\infty$ define the matrix
\begin{equation}
\left( \begin{array}{ccccc}
b_1 & a_1 &0 & 0 &\ldots \\
a_1 & b_2 & a_2 & 0& \ldots \\
0& a_2 & b_2 & a_3 & \ldots \\
\vdots & \vdots & \vdots & \vdots&\ddots \\
\end{array} \right),
\end{equation}
where $\mu_{K(\gamma)}$ is is the spectral measure for the unit vector $\delta_1$ and the self-adjoint operator
$J$ on $l_2(\mathbb{N}),$ which is defined by this matrix. We are interested in the analysis of
asymptotic behavior of  $(a_n)_{n=1}^{\infty}.$ Since $\mu_{K(\gamma)}$ is symmetric about $x=1/2$,
all $b_n$ are equal to $1/2.$ It is known (see e.g. \cite{totik1}) that $a_n>0,\,\,||Q_n|| = a_1\cdots a_n,$
which, in turn, is the reciprocal to the leading coefficient of the orthonormal polynomial of degree $n$.

In the next lemmas we use the equality $\int Q_n\,Q_m\,Q_{n+m}\,d\mu_{K(\gamma)}=||Q_{n+m}||^2,$ which follows
by orthogonality of $Q_{n+m}$ to all polynomials of smaller degree.

\begin{lemma} \label{Jac1} For all $s\in\mathbb{N}_0$ and $k\in\mathbb{N}$ we have
$$ Q_{2^s(2k+1)}= Q_{2^s}\cdot Q_{2^{s+1}k}- \frac{\|Q_{2^{s+1}k}\|^2}{\|Q_{2^s(2k-1)}\|^2}\,\,Q_{2^s(2k-1)}.$$
\end{lemma}
\begin{proof} Consider the polynomial
$P=Q_{2^s}\cdot Q_{2^{s+1}k}- \frac{\|Q_{2^{s+1}k}\|^2}{\|Q_{2^s(2k-1)}\|^2}Q_{2^s(2k-1)}.$  It is a monic
polynomial of degree $2^s(2k+1)$. Let us show that $P$ is orthogonal to $Q_n$ for all $n$ with
$0\leq n < 2^s(2k+1)$. This will mean that $P=Q_{2^s(2k+1)}.$

Suppose  $0\leq n <2^s(2k-1)$. Then orthogonality follows by comparison of the degrees.

If $n=2^s(2k-1)$ then $\int P\,Q_n\,d\mu_{K(\gamma)}=0$ due to the choice of coefficient of the addend
 in $P$ and the remark above.

Let $2^s(2k-1)<n<2^s(2k+1)$. We show that $\int Q_{2^s}\,Q_{2^{s+1}k}\,Q_n\,d\mu_{K(\gamma)}=0.$
We write $k$ in the form $k=2^q(2l+1)$ with some $q, l \in \mathbb{N}_0.$ In turn, $n=2^m(2p+1)$ with
$m\ne s.$ By Corollary \ref{cor3}, $Q_{2^{s+1}k}$ is a linear combination of products of $Q_{2^{s_j}}$ with
$\min s_j= s+1+q$ in every product. Similarly for $ Q_n,$ but here the smallest degree
is $2^m$. Therefore, $Q_{2^s}\,Q_{2^{s+1}k}\,Q_n$ is a linear combination of $A-$polynomials
and for each $A-$polynomial the exponent of the smallest term is $1.$ By Lemma \ref{int1}(d), the corresponding integral
 is zero.
\end{proof}

\begin{lemma}\label{Jac2} For all $s\in\mathbb{N}_0$ and $k\in\mathbb{N}$ we have
$$a_{2^s(2k+1)}^2\, a_{2^s(2k+1)-1}^2\cdots a_{2^{s+1}k+1}^2+a_{2^{s+1}k}^2\,a_{2^{s+1}k-1}^2\cdots a_{2^{s+1}k-2^s+1}^2= \|Q_{2^s}\|^2.$$
\end{lemma}
\begin{proof}
By Lemma \ref{Jac1} and the remark above,
\begin{equation}\label{int}
||Q_{2^s(2k+1)}||^2=\int Q_{2^s}^2\,Q_{2^{s+1}k}^2 d\mu_{K(\gamma)}-\frac{\|Q_{2^{s+1}k}\|^4}{\|Q_{2^s(2k-1)}\|^2}. \end{equation}
Let us show that
$$ \int Q_{2^s}^2\,Q_{2^{s+1}k}^2 d\mu_{K(\gamma)}= ||Q_{2^s}||^2 ||Q_{2^{s+1}k}||^2.$$

If $k=2^m,$ we have  this immediately, by Lemma \ref{int1}(a).

Otherwise, $2^{s+1}k=2^m(2l+1)$ with $l\in\mathbb{N}$ and $m\geq s+1.$ Then, by Corollary \ref{cor3},
 $Q_{2^{s+1}k}$ is a linear combination of products $Q_{2^{s_q}}\cdots Q_{2^{s_j}} \cdots Q_{2^m}$
with $s_j>m$ except the last term. From here, $Q_{2^{s+1}k}^2= Q_{2^m}^2 \cdot\sum\alpha_j\,A_j,$
where $\sum\alpha_j\,A_j$ is a linear combination of $A-$type polynomials with $s_1>m$ for each
$A_j.$ Therefore,
$$||Q_{2^{s+1}k}||^2 =\sum \alpha_j \int A_j Q_{2^m}^2 \,d\mu_{K(\gamma)}.$$

On the other hand,
$$\int Q_{2^s}^2\,Q_{2^{s+1}k}^2 d\mu_{K(\gamma)}=\sum \alpha_j \int A_j Q_{2^m}^2 \,Q_{2^s}^2\,d\mu_{K(\gamma)}.$$
By Corollary \ref{cor2}, this is $||Q_{2^s}||^2 ||Q_{2^{s+1}k}||^2.$

Therefore, \eqref{int} can be written as
$$\frac{\|Q_{2^s(2k+1)}\|^2}{\|Q_{2^{s+1}k}\|^2}=
\|Q_{2^{s}}\|^2 - \frac{\|Q_{2^{s+1}k}\|^2}{\|Q_{2^s(2k-1)}\|^2},$$
which is the desired result, as $a_n=\|Q_n\|\mathbin{/}{\|Q_{n-1}\|}.$
\end{proof}

\begin{theorem}\label{jacjac} The recurrence coefficients ${(a_n)}_{n=1}^\infty$ can be calculated recursively by using Lemma \ref{Jac2} and \eqref{norm1}.
\end{theorem}
\begin{proof} We already have $a_1=||Q_1||$ and
$a_2=||Q_2||\mathbin{/}{||Q_1||}$.
Suppose, by induction, that all $a_i$ are given up to $i=n$. If $n+1=2^s>2$  then
\begin{equation}\label{rec1}
a_{n+1}=\frac{||Q_{2^s}||}{||Q_{2^{s-1}}|| \cdot a_{2^{s-1}+1}\cdot a_{2^{s-1}+2}\cdots a_{2^{s}-1}},
\end{equation}
where the norms of polynomials can be found by  \eqref{norm1}.

Otherwise, $n+1=2^s(2k+1)$ for some $s\in\mathbb{N}_0$ and $k\in\mathbb{N}.$ By Lemma \ref{Jac2}, we have
\begin{equation}\label{rec2}
 a_{n+1}^2= a_{2^s(2k+1)}^2=\frac{\|Q_{2^s}\|^2-a_{2^{s+1}k}^2\cdots a_{2^{s+1}k-2^s+1}^2}
{ a_{2^s(2k+1)-1}^2\cdots a_{2^{s+1}k+1}^2},
\end{equation}
provided $s\ne 0.$ If $s=0$ then the denominator in the fraction above is absent.
This gives $a_{n+1},$ since the recurrence coefficients are positive.
\end{proof}

In order to  illustrate the theorem, let us consider the cases of small $s$.\\
If $s=0$ then $n+1=2k+1$ and $a_{2k+1}^2=a_1^2-a_{2k}^2.$ Next, for $s=1$ and $s=2$,
$$ a_{4k+2}^2=\frac{ ||Q_{2}||^2-a_{4k}^2\,a_{4k-1}^2}{a_{4k+1}^2}, \,\,
a_{8k+4}^2=\frac{ ||Q_{4}||^2-a_{8k}^2\,a_{8k-1}^2\,a_{8k-2}^2\,a_{8k-3}^2}{a_{8k+3}^2\,a_{8k+2}^2\,a_{8k+1}^2},
\,\,\,\mbox {etc}. $$
Thus, $a_1=\frac{\sqrt{1-2\gamma_1}}{2},\, a_2= \frac{\sqrt{1-2\gamma_2}}{\sqrt{1-2\gamma_1}}\,\gamma_1,\,
a_3^2= a_1^2-a_{2}^2,\,\, a_4=\frac{\gamma_1 \gamma_2\,\sqrt{1-2\gamma_3}}{ a_3\,\sqrt{1-2\gamma_2}}, \,\, a_5^2=a_1^2-a_4^2,$ etc.\\

\begin{remark}
If $\gamma_n<1/4$ for $1\leq n\leq s$ and $\gamma_n=1/4$ for $n>s$ then
$K(\gamma)= E_s =(2P_{2^s}/r_s+1)^{-1}[-1,1]$. Here $(P_{2^n}+r_n/2)_{n=0}^{\infty}$
are the Chebyshev polynomials for $E_s$, as is easy to check. Therefore Theorems \ref{main1} and \ref{jacjac}
are applicable for this case as well. For further information about Jacobi parameters corresponding equilibrium measures of polynomial inverse images, we refer the reader to the article \cite{van assche}.
\end{remark}
\begin{remark}
Suppose $\gamma_n=1/4$ for $n\leq N$ with $2^s \leq N < 2^{s+1}.$ Then
$a_1=1/\sqrt 8$ and $a_2=a_3=\cdots = a_{2^{s+1}-1}=1/4.$ In particular, if $\gamma_n=1/4$ for all $n$
then $a_n=1/4$ for all $n\geq 2,$ which corresponds to the case of the Chebyshev polynomials
of first kind on $[0,1].$
\end{remark}

\begin{lemma}\label{Jac3}
 Suppose $\gamma_s\leq 1/6$ for all $s.$ For fixed $s\in\mathbb{N}_0,$ let $c=\frac{4\gamma_{s+1}^2}{(1-2\gamma_{s+1})^2}$ and $C=\frac{2}{1+\sqrt{1-4c}}$.
 Then the following inequalities hold with $k\in\mathbb{N}_0$:
\begin{enumerate}[label=(\alph*)]
\item If $n=2^s(2k+1)$  then
$$ \frac{1}{2}||Q_{2^s}||^2 \leq C^{-1}\,||Q_{2^s}||^2 \leq a_n^2\cdots a_{n-2^s+1}^2 \leq ||Q_{2^s}||^2. $$
\item If $n=2^s(2k+2)$ then
$$ a_n^2 \cdots a_{n-2^s+1}^2 \leq C\,\frac{ \|Q_{2^{s+1}}\|^2}{\|Q_{2^s}\|^2}\leq 2 \frac{ \|Q_{2^{s+1}}\|^2}{\|Q_{2^s}\|^2}. $$
\end{enumerate}
\end{lemma}

\begin{proof}Note that, if $\gamma_{s+1}$ increases from $0$ to $1/6,$ then $c$ increases from $0$ to $1/4$
and $C$ increases from $1$ to $2$. By \eqref{norm1} and the definition of $r_s$, we get
\begin{equation}\label{c}
\frac{ \|Q_{2^{s+1}}\|^2}{\|Q_{2^s}\|^2 }=\gamma_{s+1}^2\, r_s^2\,\,\frac{1-2\gamma_{s+2}}{1-2\gamma_{s+1}}
=(1-2\gamma_{s+2}) \, c \, \|Q_{2^s}\|^2 < \|Q_{2^s}\|^2/4.
\end{equation}

We proceed by induction. For a fixed $s\in\mathbb{N}_0,$ let $k=0$. Then we have at once
 $$a_{2^s}^2\cdots a_{1}^2=\|Q_{2^s}\|^2\,\,\,\,\,  \mbox{ and} \,\,\,\,\,
 a_{2^{s+1}}^2\cdots a_{2^{s}+1}^2= \frac{ \|Q_{2^{s+1}}\|^2}{\|Q_{2^s}\|^2 }.$$

 Suppose $(a), \,(b)$ are satisfied for $k\leq m.$ We apply Lemma \ref{Jac2} with $k=m+1:$
 $$  a_{2^s(2m+3)}^2\cdots a_{2^s(2m+2)+1}^2 +  a_{2^s(2m+2)}^2\cdots a_{2^s(2m+2)-2^s+1}^2 = \|Q_{2^s}\|^2,$$
 where for the addend we can use $(b)$ for $k=m.$ Therefore,
$$\|Q_{2^s}\|^2-C\frac{ \|Q_{2^{s+1}}\|^2}{\|Q_{2^s}\|^2 } \leq a_{2^s(2m+3)}^2\cdots a_{2^s(2m+2)+1}^2 \leq
\|Q_{2^s}\|^2,$$
which is $(a)$ for  $k=m+1,$ by \eqref{c}.\\

Next, we claim that
\begin{equation}\label{means}
a_{2^s(2m+4)}^2\cdots a_{2^s(2m+2)+1}^2\leq \|Q_{2^{s+1}}\|^2
\end{equation}
for $m\in\mathbb{N}_0$. If $m=2l+1$ then we use Lemma \ref{Jac2} with $s+1$
instead of $s:$
$$a_{2^{s+1}(2k+1)}^2\cdots a_{2^{s+2}k+1}^2+ \,\,\,\,positive \,\,\,term\,\,\,=\|Q_{2^{s+1}}\|^2,$$
which implies  \eqref{means}, if we take $k=l+1,$  as $2(2k+1)=2m+4,\,\,4k=2m+2.$\

Suppose $m$ is even. Lemma \ref{Jac2} now gives
$$positive \,\,\,term\,\,\, + a_{2^{s+2}k}^2\cdots a_{2^{s+2}k-2^{s+1}+1}^2=\|Q_{2^{s+1}}\|^2,$$
where we take $k=m/2+1.$ Thus,  \eqref{means}  holds true in  both cases.

 Putting together $(a)$ for  $k=m+1$ and  \eqref{means} we get $(b)$ for   $k=m+1$ .
\end{proof}

\begin{theorem}\label{peri1}Let $\gamma_s\leq 1/6$ for all $s$. Then $\displaystyle\lim_{s\rightarrow \infty} a_{j\cdot 2^s+n}=a_n$ for $j\in\mathbb{N}$ and $n\in\mathbb{N}_0$.
Here, $a_0:=0.$ In particular, $\liminf a_n=0.$
\end{theorem}

\begin{proof} We first show that $\displaystyle\lim_{s\rightarrow \infty} a_{j\cdot 2^s}=0$
for all $j\in\mathbb{N}.$
Let $j=2^l(2k+1)$ where $k,l\in\mathbb{N}_0.$ For $s>0$, the Jacobi parameters admit the following
inequality by Lemma \ref{Jac3}(a):
\begin{equation}\label{ineq2}
a_{2^{s+l}(2k+1)}^2\ldots a_{2^{s+l+1}k+1}^2\leq \|Q_{2^{s+l}}\|^2.
\end{equation}
If $i<s+l$ where $i\in\mathbb{N}_0$, we have $2^{s+l}(2k+1)-2^i= 2^i(2^{s+l-i}(2k+1)-1).$
Since $2^{s+l-i}(2k+1)-1$ is a positive odd number, by  Lemma \ref{Jac3}(a), we have the inequalities
$$
\frac{1}{2} \|Q_{2^i}\|^2 \leq a_{2^{s+l}(2k+1)-2^i}^2\ldots a_{2^{s+l}(2k+1)-2^{i+1}+1}^2
\,\,\,\mbox {for} \,\,\,i=0,\ldots, s+l-1.
$$
We multiply these $s+l$ inequalities side by side:
$$ 2^{-s-l} ||Q_1||^2 \cdots ||Q_{2^{s+l-1}}\|^2\leq a_{2^{s+l}(2k+1)-1}^2\cdots a_{2^{s+l+1}k+1}^2$$
and use \eqref{ineq2}:
$$a_{j\cdot 2^s}^2 = a_{2^{s+l}(2k+1)}^2 \leq \frac{2^{s+l}\|Q_{2^{s+l}}\|^2}{\|Q_{2^{s+l-1}}\|^2\|Q_{2^{s+l-2}}\|^2\ldots \|Q_1\|^2}.$$
By \eqref{c}, the fraction above is  bounded by $2^{-s-l+2}.$ Thus, $\displaystyle\lim_{s\rightarrow \infty} a_{j\cdot 2^s}=0.$\\

If $n=1$ then $a_{j\cdot 2^s+1}^2=a_1^2-a_{j\cdot 2^s}^2 \to a_1^2,$ which is our claim.\\

Suppose, by induction, that $\displaystyle\lim_{s\rightarrow \infty} a_{j\cdot 2^s+n}=a_n$ for $n=0,1,\ldots,m$
 and  all $j\in\mathbb{N}.$ Let $m+1= 2^p(2q+1)$ where $p,q\in\mathbb{N}_0.$ If $q=0$ then
 $j\cdot 2^s+m+1=j\cdot  2^{s-p}+1,$ so we get the case with $n=1$. Thus, we can suppose  $q\in\mathbb{N}.$
 Then $j\cdot 2^s+m+1=2^p(2^{s+l-p}(2k+1)+2q+1)$ and, for large enough $s$,  we can apply Lemma \ref{Jac2}:
$$a_{j\cdot 2^s+m+1}^2\,a_{j\cdot 2^s+m}^2\cdots a_{j\cdot 2^s+m-2^p+1}^2+a_{j\cdot 2^s+m-2^p}^2\cdots a_{j\cdot 2^s+m+1-2^{p+1}}^2= \|Q_{2^p}\|^2.$$

Here all indices, except the first, are of the form $j\cdot 2^s+n$ with $n<m+1.$ Therefore, by
induction hypothesis, $a_{j\cdot 2^s+n}^2\to a_n$ as $s\to \infty$ and
$$ (\lim_{s\rightarrow \infty} a_{j\cdot 2^s+m+1}^2)\,a_{m}^2\cdots a_{m-2^p+1}^2+a_{m-2^p}^2\cdots
a_{ m+1-2^{p+1}}^2= \|Q_{2^p}\|^2.$$
On the other hand, if we apply  Lemma \ref{Jac2} to the number $m+1,$ then we get the same equality
with $a_{m+1}^2$ instead of $\lim_{s\rightarrow \infty} a_{j\cdot 2^s+m+1}^2$. Since all $a_k$ are
positive, we have the desired result.
\end{proof}
\begin{remark} Since $\liminf a_n=0,$ by \cite{dombrowski}, the set $K(\gamma)$ with $\gamma_s\leq 1/6$  has zero Lebesgue measure.
\end{remark}

\section{Widom factors}
A finite Borel measure $\mu$ supported on a non-polar compact set $K\subset\mathbb{C}$ is said to be regular in the Stahl-Totik sense if $\displaystyle \lim_{n\rightarrow\infty}\|Q_n\|^{\frac{1}{n}}= Cap(K)$ where $Q_n$ is the monic orthogonal polynomial of degree $n$ corresponding to $\mu$. It is known (see, e.g. \cite{simon1,widom}) that the equilibrium measure is regular in the Stahl-Totik sense. While $\|Q_n\|^{\frac{1}{n}}\mathbin{/}Cap(K)$ has limit $1$, the ratio  $W_n=\|Q_n\|\mathbin{/}(Cap(K))^n$  may have various asymptotic behavior.
We call $W_n$ the {\it Widom factor} due to the paper \cite{widom2}. These values play an important role in spectral theory of orthogonal polynomials on several intervals.

 Let $\displaystyle E=[\alpha,\beta]\setminus \bigcup_{i=1}^n(\alpha_i,\beta_i)$ where $\alpha,\beta \in \mathbb{R}$ and the intervals $(\alpha_i,\beta_i)$ are disjoint subsets of $[\alpha,\beta]$. Let $d\mu=f(t)dt$ on $E$ and $(a_n)_{n=1}^{\infty}$ be the Jacobi parameters corresponding to $\mu.$ Then \cite{Chris}
\begin{align}\label{sze}
\int \log{f(t)}d\mu_E(t)>-\infty \iff \limsup_{n\rightarrow\infty}\frac{a_1\ldots a_n}{Cap(E)^n}>0.
\end{align}
 For further generalizations and different aspects of this result, see \cite{christiansen,Chris,Chris2,Peherstorfer,Sodin}.

We already know that $Cap(K(\gamma))=\exp{(\sum_{k=1}^{\infty}2^{-k}\log{\gamma_k})}.$  In terms of  \\$(\gamma_k)_{k=1}^\infty$ we can rewrite $\|Q_{2^s}\|$ as \begin{equation}\label{eq1}
 \frac{\sqrt{1-2\gamma_{s+1}}}{2}\,\exp{\left(2^s\sum_{k=1}^{s}2^{-k}\log{\gamma_k}\right)}.
\end{equation}
 Therefore, 
 \begin{equation}\label{eq2}
\displaystyle W_{2^s}= \frac{\sqrt{1-2\gamma_{s+1}}}{2\exp{\left(\sum_{k=s+1}^{\infty}2^{s-k}\log{\gamma_k}\right)}}\geq \sqrt{2},
\end{equation}
since $\gamma_s\leq 1/4.$ The limit values $\gamma_s = 1/4$ for all $s$ give the Widom Factors for
the equilibrium measure on $[0,1]$.

Clearly, \eqref{eq2} implies that $\limsup{W_n}>0.$ If $\gamma_s \leq 1/6$ for all $s$ then
\begin{equation}\label{eq3}
 W_{2^s}\geq \sqrt{6}.
\end{equation}
Let us show that, in this case, $\liminf{W_n}>0.$
\begin{theorem}Let $(W_n)_{n=1}^\infty$ be Widom factors for $\mu_{K(\gamma)}$ where $\gamma_s\leq 1/6$
 for all $s$. Then
\begin{enumerate}[label=(\alph*)]
\item $\displaystyle \,\,\liminf_{s\rightarrow\infty}W_{2^s}=\liminf_{n\rightarrow\infty}W_{n}.$\\
\item$\displaystyle \,\,\limsup_{n\rightarrow\infty}{W_n}=\infty$.
\end{enumerate}
\end{theorem}
\begin{proof}\begin{enumerate}[label=(\alph*)]
\item We show that $W_n > W_{2^s}$ for $2^{s} < n<2^{s+1}$. Let $n=2^{s}+2^{s_1}+\ldots+2^{s_m}$
with  $s>s_1>s_2>\ldots>s_m\geq 0.$ Then we decompose $a_1 \cdots a_n$ into
groups $(a_1 \cdots a_{2^s}) \cdot ( a_{2^s+1}\cdots a_{2^s+2^{s_1}}) \cdots 
(a_{2^s+\cdots +2^{s_{m-1}}+1} \cdots a_n).$ For the first group we have
$a_1 \cdots a_{2^s}=||Q_{2^s}||.$ For the second group we use Lemma \ref{Jac3}(a) with 
$n=2^s+2^{s_1}: a_{2^s+1}\cdots a_{2^s+2^{s_1}} \geq ||Q_{2^{s_1}}|| /\sqrt 2.$ Similar estimation is valid for
all other groups. Therefore,
\begin{eqnarray*}W_n&=& \frac{a_1 \cdots a_{2^s}}{Cap(K(\gamma)^{2^s}}\,\,
\frac{ a_{2^s+1}\cdots a_{2^s+2^{s_1}}}{Cap(K(\gamma)^{2^{s_1}}} \cdots 
\frac{a_{2^s+\cdots +2^{s_{m-1}}+1} \cdots a_n}{Cap(K(\gamma)^{2^{s_m}}}\\
&\geq& W_{2^s} W_{2^{s_1}}\cdots W_{2^{s_m}} (\sqrt 2)^{-m},
\end{eqnarray*}
which exceeds $W_{2^s} (\sqrt 3)^{m}, $ by \eqref{eq3}. From here, 
$\min_{2^{s} \leq n<2^{s+1}}W_n=W_{2^s}$ and the result follows.\\

\item Applying the procedure above to $W_{2^s-1}$ and taking the limit gives the desired result. 
\end{enumerate}
\end{proof}

In order to illustrate the behavior of Widom factors, let us consider some examples. Suppose $\gamma_s\leq 1/6$
 for all $s$. 
 
\begin{example} If $\gamma_s \to 0$ then $W_n\to \infty.$ Indeed, 
$W_{2^s}\geq \frac{1}{\sqrt 6}\,\exp(\frac{1}{2}\,\log \frac{1}{\gamma_{s+1}}).$ 
\end{example}
\begin{example} There exists $\gamma_s \nrightarrow 0$ with $W_n\to \infty.$ 
Indeed, we can take $\gamma_{2k}=1/6, \,\gamma_{2k-1}=1/k.$
\end{example}
\begin{example} If $\gamma_s \geq c > 0$  for all $s$ then $\liminf_{n\rightarrow\infty} W_n\leq 1/2c$.
\end{example}
\begin{example} There exists $\gamma$ with $\inf \gamma_s=0$ and 
$\liminf_{n\rightarrow\infty} W_n <\infty.$
Here we can take $\gamma_s=1/6$ for $s\ne s_k$ and $\gamma_{s_k}=1/k$ for a sparse sequence
$(s_k)_{k=1}^{\infty}.$ Then $(W_{2^{s_k}})_{k=1}^{\infty}$ is bounded.
\end{example}

\section{Towards the Szeg\H{o} class}

The convergence of the integral on the left-hand side of \eqref{sze} defines
the Szeg\H{o} class of spectral measures for the finite gap Jacobi matrices. The Widom
condition on the right-hand side is the main candidate to characterize the
Szeg\H{o} class for the general case, see \cite{christiansen,Peherstorfer,Sodin}.

The equilibrium measure is the most natural measure in the theory of orthogonal polynomials.
In particular, for known examples, the values $\limsup W_n$ associated with equilibrium measures 
are bounded below by positive numbers. So we make the following conjecture:
\begin{Conjecture}If a compact set $K\subset\mathbb{R}$ is regular with respect to the Dirichlet problem then
the Widom condition $W_n  \nrightarrow 0$ is valid for the equilibrium measure $\mu_K.$
\end{Conjecture}

Concerning the Szeg\H{o} condition, one can conjecture
that, in the case of measures with non-polar perfect support $K$,  the left-hand side of \eqref{sze}
can be written as
\begin{equation}\label{Sz}
I(\mu):=\int \log (d\mu/d\mu_K)d\mu_K(t)>-\infty .
\end{equation}
Indeed, for the finite gap case, this coincides with the condition in \eqref{sze}, since  the integral
$\int \log (d\mu_K/dt)d\mu_K(t)$ converges. By Jensen's inequality (see also Section 4 in \cite{Chris}), the
value  $I(\mu)$ is nonpositive and it attains its maximum $0$ just in the case 
$\mu=\mu_K$ a.e. with respect to $\mu_K$.
On the other hand, there are strong objections to \eqref{Sz}, based on the numerical evidence
from \cite{kruger}, where, for the Cantor-Lebesgue measure $\mu_{CL}$ on the classical Cantor set $K_0$,
the Jacobi parameters $(a_n)$ were calculated for $n\leq 200.000.$ The Widom factors for such values
behave as a bounded below (by a positive number) sequence. Therefore, if we wish to preserve the Widom characterization of the Szeg\H{o} class, the integral $I(\mu_{CL})$ must converge, but, since $\mu_{CL}$ and
$\mu_{K_0}$ are mutually singular, it is not the case.

\end{document}